\documentclass[11pt]{article}

\usepackage{amsmath,amsfonts,amssymb,enumerate}
\usepackage{bbm}
\usepackage{latexsym}
\usepackage{graphicx}
\usepackage{algpseudocode,algorithm}
\usepackage[margin=1in]{geometry}
\usepackage{epstopdf}
\usepackage[margin=0.5cm]{subcaption}

\newtheorem{theorem}{\bf Theorem}[section]
\newtheorem{lemma}[theorem]{\bf Lemma}

\newtheorem{example}[theorem]{\bf Example}

\newenvironment{proof}{\noindent{\em Proof:}}{\quad \hfill$\Box$\vspace{2ex}}

\title{Finding Dantzig selectors with a proximity operator based fixed-point algorithm \thanks{Cleared for public release by WPAFB Public Affairs on 09 Oct 2013.  Case Number: 88ABW-2013-4324.  This research is supported in part by an award from National Research Council via the Air Force Office of Scientific Research and by the US National Science Foundation under grant DMS-1115523.}}

\author{
Ashley Prater\thanks{Air Force Research Laboratory, Information Directorate, Rome, NY 13441} \and Lixin Shen\thanks{Department of Mathematics, Syracuse
University, Syracuse, NY 13244, USA. } \and Bruce W. Suter$^\dagger$
}
\date{}

\begin{document}

\maketitle
\begin{abstract}
In this paper, we study a simple iterative method for finding the Dantzig selector, which was designed for linear regression problems. The method consists of two main stages. The first stage is to approximate the Dantzig selector through a fixed-point formulation of solutions to the Dantzig selector problem.  The second stage is to construct a new estimator by regressing data onto the support of the approximated Dantzig selector.  We compare our method to an alternating direction method, and present the results of numerical simulations using both the proposed method and the alternating direction method on synthetic and real data sets.  The numerical simulations demonstrate that the two methods produce results of similar quality, however the proposed method tends to be significantly faster.
\end{abstract}


\section{Introduction}
This paper considers the problem of estimating a vector of parameter $\beta \in \mathbb{R}^p$ from the linear problem
\begin{equation}\label{problem:linear}
y = X \beta + z,
\end{equation}
where $y \in \mathbb{R}^n$ is a vector of observations, $X$ an $n \times p$ predictor matrix, and  $z$ a vector of independent normal random variables. The goal is to find a relevant parametric vector $\beta^\star \in \mathbb{R}^p$ among many potential candidates and obtain high prediction accuracy.

The $\ell_1$ penalized least squares estimator for problem~\eqref{problem:linear} has been the focus of a great deal of attention for variable selection and estimation in high-dimensional linear regression when the number of variables is much larger than the sample size \cite{Efron-Hastie-Johnstone-Tibshirani:AS:2004,Meinshausen-Buhlmann:AS:06,Osborne-Presnell-Turlach:JCGS:00,Tibshirani:jrss:96,Tibshirani:jrss:11,Zhao-Yu:JMLR:06}. Recently the Dantzig selector was proposed for problem~\eqref{problem:linear} in \cite{Candes-Tao:AS:07}.  The Dantzig selector $\widehat{\beta} \in \mathbb{R}^p$ is a solution to the optimization problem
\begin{equation}\label{problem:Dantzig}
\widehat{\beta} \in \mathrm{argmin}\{\|\beta\|_1: \|D^{-1}X^\top (X\beta-y)\|_\infty \le \delta\},
\end{equation}
with a fixed parameter $\delta>0$ and a diagonal matrix $D$ where the diagonal entries are equal to the $\ell_2$ norm of the columns of $X$. Here, we write $\|x\|_q$ for the $\ell_q$ norm of $x\in \mathbb{R}^p$, $1 \le q \le \infty$. Optimal $\ell_2$ rate properties for $\|\widehat{\beta} -\beta^\star\|_2$ were established under a sparsity scenario and impressive empirical performance on real world problems involving large values of $p$ was shown in \cite{Candes-Tao:AS:07}. Since then the Dantzig selector has received a considerable amount of attention. Discussions on the Dantzig selector can be found in  \cite{Bickel:AS:07,Cai-Lv:AS:07,Candes-Tao:AS:07b,Efron-Hastie-Tibshirani:AS:07,Friedlander-Saunders:AS:07,Meinshausen-Rocha-Yu:AS:07,Riotv:AS:07}.  In \cite{James-Radchenko-Lv:jrss:09}, an algorithm was proposed for fitting the entire coefficient path of the Dantzig selector with a similar computational cost to the least angle algorithm that is used to compute the $\ell_1$ minimization via the LASSO technique. The Dantzig selector is a convex, but not strictly convex, optimization problem. Unique solutions are in general not guaranteed. Conditions ensuring the uniqueness of the Dantzig selector were presented in \cite{Dicker-Lin:CJS:12}. In \cite{Li-Dicker-Zhao:SS:12} a new class of Dantzig selectors for linear regression problems for right-censored outcomes was proposed.

The importance of the Dantzig selector in linear regressions has been demonstrated in the aforementioned work. Efficient methods for solving  problem~\eqref{problem:Dantzig}, which however were not emphasized in the current literature, are highly needed. In \cite{Candes-Tao:AS:07}, the problem is cast as a linear program which is solved by using a primal-dual interior point algorithm~\cite{Body-Vandenberghe:04}. As it is well known, interior point methods are not efficient for large-scale problems. In \cite{Becker-Candes-Grant:MPC:10}, the problem is cast as linear cone programming problem for which a smooth approximation to its dual problem is solved by an optimal first-order method \cite{Beck-Teboulle:SIAMIS:09,Nesterov:MP:05}. Recently, an alternating direction method (ADM) for finding the Dantzig selector was studied in \cite{Lu-Pong-Zhang:CSDA:12}. Numerical experiments showed that this method usually outperforms the method in \cite{Becker-Candes-Grant:MPC:10} in terms of CPU time while producing solutions of comparable quality.   The problem  was rewritten in \cite{Lu-Pong-Zhang:CSDA:12} in a form to which ADM can be easily applied.  ADM itself is an iterative algorithm.  In each iterate, two subproblems are needed to be solved successively.  One of the subproblems has a closed form solution, while the other does not and is approximated by a nonmonotone gradient method proposed in \cite{Lu-Zhang:MP:12}.  To alleviate the difficulty caused by the subproblem without a closed form solution, a linearized ADM was proposed for the Dantzig selector and was shown to be efficient for solving both synthetic and real world data sets in \cite{Wang-Yuan:SISC:12}.

In this paper, the Dantzig selectors for problem~\eqref{problem:Dantzig} are found by an algorithm based upon proximity operators. We first rewrite the problem as an unconstrained structural optimization problem via an indicator function. The resulting problem is then solved by a primal-dual algorithm. In comparison with the one given in \cite{Lu-Pong-Zhang:CSDA:12}, our proposed algorithm is easy to implement.  Ours achieves comparable quality results while consuming much less CPU time.

The outline of the paper is organized as follows. In Section \ref{sec:alg} we present our fixed-point theory based proximity operator algorithm for solving problem \eqref{problem:Dantzig}.  In Section \ref{sec:experiments}, we present numerical experiments comparing the accuracy and efficiency of the proposed algorithm with ADM proposed in~\cite{Lu-Pong-Zhang:CSDA:12}.  The first set of experiments uses simulated sparse signals and the second set uses samples of biomarker data to predict the diagnosis of leukemia patients.  Section~\ref{sec:conclusion} concludes the paper.

The following notation will be used in the rest of the paper.  For any vector $u \in \mathbb{R}^d$, let $u_i$ and $u(i)$ both denote the $i$-th component of $u$. Also for any vector $u\in\mathbb{R}^d$, $|u|$ is the component-wise absolute values of $u$, that is the $i$-th component of $|u|$ is $|u_i|$, while $\mathrm{sign}(u)$ is the vector whose $i$-th component is $1$ if $u_i>0$ and $-1$ otherwise. Given two vectors $u$ and $v$ in $\mathbb{R}^d$,  $x \circ y$ denotes the Hadamard (component-wise) product of $u$ and $v$, $\max\{u,v\}$ denotes the vector whose $i$-th entry is $\max\{u_i,v_i\}$, and $\min\{u,v\}$ denotes the vector whose $i$-th entry is $\min\{u_i,v_i\}$. Let $\mathbbm{1}$ denote the vector of all ones whose dimension should be clear from the context.

The natural numbers are given by $\mathbb{N}$.  For the  usual $d$-dimensional Euclidean space denoted by $\mathbb{R}^d$  we define $\langle x, y \rangle:=\sum_{i=1}^d x_i y_i$, for $x, y \in \mathbb{R}^d$,  the standard inner product in $\mathbb{R}^d$. We denote by $\|\cdot\|_1$, $\|\cdot\|_2$, and $\|\cdot\|_\infty$ the $\ell_1$ norm, $\ell_2$ norm, and the $\ell_\infty$ norm of a vector, respectively.  The class of all lower semicontinuous convex functions $f: \mathbb{R}^d \rightarrow (-\infty, +\infty]$ such that $\mathrm{dom} f:=\{x \in \mathbb{R}^d: f(x) <+\infty\} \neq \emptyset$ is denoted by $\Gamma_0(\mathbb{R}^d)$. For a closed convex set $\mathcal{C}$ of  $\mathbb{R}^d$, its indicator function $\iota_{\mathcal{C}}$ is in $\Gamma_0(\mathbb{R}^d)$ and is defined as
$$
\iota_\mathcal{C}(u): =\left\{
               \begin{array}{ll}
                 0, & \hbox{if $u\in \mathcal{C}$,} \\
                 +\infty, & \hbox{otherwise.}
               \end{array}
             \right.
$$
For a function $f \in \Gamma_0(\mathbb{R}^d)$, $\mathrm{argmin}_{x \in C} f(x)$ is the set of points of the given argument in $C$ for which $f$ attains its minimum value, i.e., $\mathrm{argmin}_{x \in C} f(x)=\{x\in C: f(y) \ge f(x) \; \mbox{for all} \; y \in C\}$.

\section{The Dantzig Selector with Proximity Algorithms}\label{sec:alg}

In this section, we develop a proximity algorithm for solving the optimization problem~\eqref{problem:Dantzig}. We begin with reviewing two existing works on this problem, namely the alternating direction method (ADM) proposed in \cite{Lu-Pong-Zhang:CSDA:12} and the linearized alternating direction method of multipliers (LADM) proposed in \cite{Wang-Yuan:SISC:12}. Both methods work on the reformulated optimization problem~\eqref{problem:Dantzig} with $D=I$ as follows:
\begin{equation}\label{problem:reformulation}
\min_{\beta\in \mathbb{R}^p, \tau \in \{\tau: \|\tau\|_\infty \le \delta\}} \{\|\beta\|_1: X^\top(X\beta-y)=\tau\},
\end{equation}
where $\tau \in \mathbb{R}^p$ is an auxiliary variable.  The augmented Lagrangian function for problem~\eqref{problem:reformulation} is
\begin{equation*}
L_c(\beta, \tau, \gamma):=\|\beta\|_1+\langle \gamma, X^\top(X\beta-y)-\tau\rangle +\frac{c}{2}\|X^\top(X\beta-y)-\tau\|_2^2,
\end{equation*}
where $\gamma \in \mathbb{R}^p$ is the Lagrange multiplier and $c>0$ is a penalty parameter.

The iterative scheme of ADM for optimization problem~\eqref{problem:reformulation} is
\begin{equation*}\label{scheme:ADM-1}
\left\{
  \begin{array}{l}
    \tau^{k+1} \leftarrow \mathrm{argmin}_{\tau \in \{\tau: \|\tau\|_\infty \le \delta\}} L_c(\beta^{k}, \tau, \gamma^k),\\
    \beta^{k+1} \leftarrow \mathrm{argmin}_{\beta \in \mathbb{R}^p} L_c(\beta, \tau^{k+1}, \gamma^k),\\
    \gamma^{k+1}=\gamma^k+c (X^\top(X\beta^{k+1}-y)-\tau^{k+1}),
  \end{array}
\right.
\end{equation*}
which, with some elementary manipulations, can equivalently be written as
\begin{equation}\label{scheme:ADM-2}
\left\{
  \begin{array}{l}
    \tau^{k+1} \leftarrow \mathrm{argmin}_{\tau \in \{\tau: \|\tau\|_\infty \le \delta\}} \|\tau-(X^\top(X\beta^{k}-y)+\frac{\gamma^{k}}{c})\|_2^2,\\
    \beta^{k+1} \leftarrow \mathrm{argmin}_{\beta \in \mathbb{R}^p} \{\|\beta\|_1+\frac{c}{2}\|X^\top(X\beta-y)-\tau^{k+1} + \frac{\gamma^{k}}{c}\|_2^2\},\\
    \gamma^{k+1}=\gamma^k+c (X^\top(X\beta^{k+1}-y)-\tau^{k+1}).
  \end{array}
\right.
\end{equation}
The $\tau$-related subproblem in \eqref{scheme:ADM-2} has a closed form solution, but the $\beta$-related subproblem does not and is solved approximately by using the nonmonotone gradient method in \cite{Lu-Pong-Zhang:CSDA:12}.

The iterative scheme of LADM for optimization problem~\eqref{problem:reformulation} is
\begin{equation}\label{scheme:LADM}
\left\{
  \begin{array}{l}
    \beta^{k+1} \leftarrow \mathrm{argmin}_{\beta \in \mathbb{R}^p}\{\|\beta\|_1+c\langle v^k, \beta-\beta^k\rangle + \frac{\ell}{2}\|\beta-\beta^k\|_2^2\},\\
    \tau^{k+1} \leftarrow \mathrm{argmin}_{\tau \in \{\tau: \|\tau\|_\infty \le \delta\}} \|\tau-(X^\top(X\beta^{k+1}-y)+\frac{\gamma^{k}}{c})\|_2^2,\\
    \gamma^{k+1}=\gamma^k+c (X^\top(X\beta^{k+1}-y)-\tau^{k+1}),
  \end{array}
\right.
\end{equation}
where $\ell>0$ is a proximal parameter and $v^k:=X^\top X(X^\top(X\beta^k-y)-\tau^{k} + \frac{\gamma^{k}}{c})$. Note that the order of updating $\tau^{k+1}$ and $\beta^{k+1}$ in ADM is reversed in LADM. For the $\beta$-subproblem in LADM, the last two terms in the objective function can be viewed as the linearization of the quadratic term $\frac{c}{2}\|X^\top(X\beta-y)-\tau^{k} + \frac{\gamma^{k}}{c}\|_2^2$ with respect to $\beta$ at $\beta^k$ after dropping a constant. Furthermore, the $\beta$-subproblem, after completing the square of these two terms and ignoring the resulting constant term, is the same as
\begin{equation*}
\beta^{k+1} \leftarrow \mathrm{argmin}_{\beta \in \mathbb{R}^p}\{\|\beta\|_1
+\frac{\ell}{2}\|\beta-(\beta^k-\frac{c}{\ell}v^k)\|_2^2,
\end{equation*}
which has a closed form solution. The $\tau$-subproblem has a closed form solution as in ADM. Therefore, LADM can be easily and efficiently implemented. It was shown in \cite{Wang-Yuan:SISC:12} that for any $c>0$ and $\ell> 2 \|X^\top X\|_2^2$ and any initial iterate $(\beta^0, \tau^0, \gamma^0)$, the sequence $\{(\beta^k, \tau^k, \gamma^k): k \in \mathbb{N}\}$ converges.  Furthermore, the limit of the sequence $\{(\beta^k, \tau^k): k \in \mathbb{N}\}$ is a solution of the Dantzig selector problem \eqref{problem:reformulation}.

In the following,  we present our fixed-point theory based proximity operator algorithm for solving the optimization problem~\eqref{problem:Dantzig}. For simplicity of exposition, with the matrices $X$ and $D$, the vector $y$, and the constant $\delta$ appearing in problem~\eqref{problem:Dantzig}, we set
\begin{equation}\label{eq:A-b-C}
A:=D^{-1} X^\top X, \quad b:=D^{-1}X^\top y, \quad \mathcal{C}:=\{\beta \in \mathbb{R}^p: \|\beta-b\|_\infty \le \delta\}.
\end{equation}

Then the optimization problem~\eqref{problem:Dantzig} can be rewritten as
\begin{equation}\label{problem:Dantzig-New}
\widehat{\beta} \in \mathrm{argmin}\{\|\beta\|_1 + \iota_{\mathcal{C}}(A \beta): \beta\in \mathbb{R}^p\}.
\end{equation}
The objective function of this problem is convex and coercive thanks to the $\ell_1$-norm being coercive. Hence a solution to problem~\eqref{problem:Dantzig-New} exists and can be characterized in terms of proximity operator. To this end, we review the definition of proximity operator.

For a function $f \in \Gamma_0(\mathbb{R}^d)$, the proximity operator of $f$ with parameter $\lambda$, denoted by $\mathrm{prox}_{\lambda f}$, is a mapping from $\mathbb{R}^d$ to itself, defined for a given point $x \in \mathbb{R}^d$ by
$$
\mathrm{prox}_{\lambda f} (x):=\mathop{\mathrm{argmin}} \left\{\frac{1}{2\lambda} \|u-x\|^2_2 + f(u): u \in \mathbb{R}^d \right\}.
$$

Now, we can present a characterization of solutions of problem~\eqref{problem:Dantzig-New} that is simply derived from Fermat's rule.
\begin{theorem}\label{thm:char}
Let the $p\times p$ matrix $A$ and the vector $b \in \mathbb{R}^p$ be given in \eqref{eq:A-b-C}. If $\beta \in \mathbb{R}^p$ is a solution to problem~\eqref{problem:Dantzig-New}, then for any $\alpha>0$ and $\lambda>0$ there exists a vector $\tau \in \mathbb{R}^p$ such that
\begin{eqnarray}\label{eq:char1}
\beta&=&\mathrm{prox}_{\frac{1}{\alpha}\|\cdot\|_1}\left(\beta-\frac{\lambda}{\alpha} A^\top \tau\right),\\
\tau&=&(I-\mathrm{prox}_{\iota_{\mathcal{C}}})(A\beta+\tau) \label{eq:char2}.
\end{eqnarray}
Conversely, if there exist $\alpha>0$ and $\lambda>0$ such that $\beta, \tau \in \mathbb{R}^p$ satisfy equations~\eqref{eq:char1} and \eqref{eq:char2}, then $\beta$ is a solution of problem~\eqref{problem:Dantzig-New}.
\end{theorem}
\begin{proof}\ \ The proof of the result follows straightforwardly a general result in \cite[Proposition 1]{Li-Micchelli-Shen-Xu:IP-12}. For completeness, we present its proof here. First, we assume that $\beta$ is a solution to problem~\eqref{problem:Dantzig-New}. By Fermat's rule and the chain rule of subdifferentiation,  $0 \in \partial \|\cdot\|_1 (\beta) + A^\top \partial \iota_\mathcal{C} (A\beta)$. Then for any $\alpha>0$ and $\lambda>0$ there exists $\tau \in \frac{1}{\lambda} \partial \iota_\mathcal{C}(A\beta)$ such that $-\frac{\lambda}{\alpha} A^\top \tau \in \partial \left(\frac{1}{\alpha}\|\cdot\|_1\right) (\beta)$, that is, in terms of proximity operator, equation~\eqref{eq:char1}. Since the set $\partial \iota_\mathcal{C}(A\beta)$ is a cone, then $\tau \in \frac{1}{\lambda} \partial \iota_\mathcal{C}(A\beta)$ implies $\tau \in \partial \iota_\mathcal{C}(A\beta)$ which is essentially equivalent to equation~\eqref{eq:char2}.

Conversely, if equations~\eqref{eq:char1} and \eqref{eq:char2} are satisfied, we then have $-\frac{\lambda}{\alpha} A^\top \tau \in \partial \left(\frac{1}{\alpha}\|\cdot\|_1\right) (\beta)$ and $\tau \in \partial \iota_\mathcal{C}(A\beta)$ accordingly. Using the fact that the set $\partial \iota_\mathcal{C}(A\beta)$ is a cone again, the second inclusion $\tau \in \partial \iota_\mathcal{C}(A\beta)$ implies $\frac{\lambda}{\alpha}\tau \in \frac{1}{\alpha}\partial \iota_\mathcal{C}(A\beta)$. Multiplying $A^\top$ to both sides of the previous inclusion and using the chain rule $\partial(\iota_\mathcal{C} \circ A)(\beta)=A^\top \partial \iota_\mathcal{C}(A\beta)$, we have that $\frac{\lambda}{\alpha} A^\top \tau \in \frac{1}{\alpha}\partial(\iota_\mathcal{C} \circ A)(\beta)$. Since $-\frac{\lambda}{\alpha}A^\top \tau \in \partial(\frac{1}{\alpha})(\beta)$, we obtain $0 \in \partial \|\cdot\|_1 (\beta) + \partial(\iota_\mathcal{C} \circ A)(\beta)$. This shows that $\beta$ is a solution to problem~\eqref{problem:Dantzig-New}.
\end{proof}

We comment on the computation of the proximity operators $\mathrm{prox}_{\frac{1}{\alpha}\|\cdot\|_1}$ and $\mathrm{prox}_{\iota_\mathcal{C}}$ appearing in equations~\eqref{eq:char1} and \eqref{eq:char2}. The proximity operator $\mathrm{prox}_{\frac{1}{\alpha}\|\cdot\|_1}$ at any $u \in \mathbb{R}^p$ is the well-known soft-thresholding operator given as follow:
\begin{equation}\label{eq:prox-L1}
\mathrm{prox}_{\frac{1}{\alpha}\|\cdot\|_1} (u) = \mathrm{sign}(u) \circ \max\left\{|u|-\frac{1}{\alpha}\mathbbm{1}, 0\right\}.
\end{equation}

\begin{lemma}\label{lemma:proj-prox}
Let $\delta$ be a constant, let $b$ be a vector in $\mathbb{R}^p$, and let the set $\mathcal{C}$ be given in \eqref{eq:A-b-C}. Then for any vector $v \in\mathbb{R}^p$,
\begin{equation}\label{eq:prox-C}
\mathrm{prox}_{\iota_\mathcal{C}}(v)=b+\min\{\max\{v-b, -\delta \mathbbm{1}\}, \delta \mathbbm{1}\}.
\end{equation}
and
\begin{equation*}
(I-\mathrm{prox}_{\iota_\mathcal{C}})(v) = \mathrm{prox}_{\delta\|\cdot\|_1} (v-b).
\end{equation*}
\end{lemma}
\begin{proof}\ \
It is well-known that the proximity operator  $\mathrm{prox}_{\iota_\mathcal{C}}$ is the projection operator onto the set $\mathcal{C}$. Since the set $\mathcal{\mathcal{C}}$ is the cube with $b$ as its center and $2\delta$ as the length of its side in $\mathbb{R}^p$.  Hence, $\mathrm{prox}_{\iota_\mathcal{C}}(v)$ the projection of the vector $v \in \mathbb{R}^p$ is given by \eqref{eq:prox-C}.
Further, it holds that $(I-\mathrm{prox}_{\iota_\mathcal{C}})(v)=(v-b) - \min\{\max\{(v-b), -\delta \mathbbm{1}\}, \delta \mathbbm{1}\}$. From this identity, we can directly check that for each $i$ from $1$ to $n$
$$
(v-b)_i - \min\{\max\{(v-b)_i, -\delta \}, \delta \}=\mathrm{sign}((v-b)_i) \cdot \max\{|(v-b)_i|-\delta,0\},
$$
which, by using equation~\eqref{eq:prox-L1}, is $\mathrm{prox}_{\delta\|\cdot\|_1} ((v-b)_i)$. This completes the proof.
\end{proof}

As a result of Lemma~\ref{lemma:proj-prox}, equation~\eqref{eq:char2} can be rewritten as follows:
\begin{equation}\label{eq:char2-new}
\tau=\mathrm{prox}_{\delta\|\cdot\|_1} (A\beta+\tau-b).
\end{equation}
Therefore, by Theorem~\ref{thm:char}, finding a solution $\beta$ to problem~\eqref{problem:Dantzig-New} amounts to solving the coupled fixed-point equations~\eqref{eq:char1} and \eqref{eq:char2-new}. 

Two iterative schemes can be derived from equations~\eqref{eq:char1} and \eqref{eq:char2-new}. Let us write equation~\eqref{eq:char1} as $\beta=\mathrm{prox}_{\frac{1}{\alpha}\|\cdot\|_1}\left(\beta-\frac{\lambda}{\alpha} A^\top (2\tau-\tau)\right)$. With any initial estimates $\tau^{-1}=\tau^0$ and $\beta^0$, the first iterative scheme based upon equations~\eqref{eq:char1} and \eqref{eq:char2-new} is as follows:
\begin{eqnarray}\label{eq:itr}
\left\{\begin{array}{l}
\beta^{k+1}= \mathrm{prox}_{\frac{1}{\alpha}\|\cdot\|_1}(\beta^{k}-\frac{\lambda}{\alpha} A^\top(2 \tau^k -\tau^{k-1})), \\
\tau^{k+1}= \mathrm{prox}_{\delta\|\cdot\|_1}(A \beta^{k+1}+\tau^k-b). \label{eq:itr-tau}
\end{array}\right.
\end{eqnarray}
We would like to comment the connection of this scheme with some existing ones. The dual formulation of \eqref{problem:Dantzig-New}, as derived in \cite{Lu-Pong-Zhang:CSDA:12}, is
$$
\max_{\tau \in \mathbb{R}^p}\{- \langle b, \tau \rangle -\delta \|\tau\|_1: \|A^\top \tau\|_\infty \le 1 \}.
$$
Applying the primal-dual hybrid gradient method (see \cite[Equation 2.18]{Esser-Zhang-Chan:SIAMIS:2010}) to the above dual formulation yields exactly the iterative scheme \eqref{eq:itr}. It was further pointed out in \cite{Chambolle-Pock:JMIV11} that the iterative scheme \eqref{eq:itr} is essentially the same as the linearized ADM applying to problem~\eqref{problem:Dantzig-New}. In other words,  the iterative scheme \eqref{eq:itr} is the same as \eqref{scheme:LADM} in the case of $D=I$.

Now, let us introduce the second iterative scheme for problem~\eqref{problem:Dantzig-New}.  Let us write equation~\eqref{eq:char2-new} as
$\tau=\mathrm{prox}_{\delta\|\cdot\|_1} (A(2\beta-\beta)+\tau-b)$. With any initial estimates $\beta^{-1}=\beta^0$ and $\tau^0$, the second iterative scheme based upon equations~\eqref{eq:char1} and \eqref{eq:char2-new} is as follows:
\begin{eqnarray}\label{eq:itr-tau-beta}
\left\{\begin{array}{l}
\tau^{k+1}= \mathrm{prox}_{\delta\|\cdot\|_1}(A (2\beta^{k}-\beta^{k-1})+\tau^k-b), \\
\beta^{k+1}= \mathrm{prox}_{\frac{1}{\alpha}\|\cdot\|_1}(\beta^{k}-\frac{\lambda}{\alpha} A^\top \tau^{k+1}).
\end{array}\right.
\end{eqnarray}

The sequence $\{(\beta^k, \tau^k): k \in \mathbb{N}\}$ generated by the iterative schemes~\eqref{eq:itr} and \eqref{eq:itr-tau-beta} will converge for any initial seeds when $\lambda/\alpha < 1/\|A\|_2^2$. The proof of this convergence result can be found in \cite{Chambolle-Pock:JMIV11,Li-Micchelli-Shen-Xu:IP-12}. Hence, the limit of the sequence $\{(\beta^k, \tau^k): k \in \mathbb{N}\}$ is a fixed-point of equations~\eqref{eq:char1} and \eqref{eq:char2}. In particular, the limit of the sequence $\{\beta^k: k \in \mathbb{N}\}$ is a solution to problem \eqref{problem:Dantzig-New}.

As noted in \cite{Candes-Tao:AS:07}, the Dantzig selector often slightly underestimates the true values of the nonzero parameters. To correct this bias and increase performance in practical settings, a postprocessing procedure was proposed in \cite{Candes-Tao:AS:07}. Assume that $\beta^\infty$ is the limit of the sequence $\{\beta^k: k \in \mathbb{N}\}$ that is generated through the iterative scheme~\eqref{eq:itr-tau-beta}. This postprocessing consists of two steps. The first step is to estimate $\Lambda:=\{i: \beta_i^\infty \neq 0\}$, the support of the vector $\beta^\infty$. Let $X_\Lambda$ be the $n \times |\Lambda|$ submatrix obtained by extracting the columns of $X$ corresponding to the indices in $\Lambda$, and let $\widehat{\beta}_\Lambda$ be the $|\Lambda|$-dimensional vector obtained by extracting the coordinates of $\widehat{\beta} \in \mathbb{R}^p$ corresponding to the indices in $\Lambda$. The second step of the postprocessing is to construct the estimator $\widehat{\beta} \in \mathbb{R}^p$ such that
\begin{equation*}
\widehat{\beta}_\Lambda =  \mathrm{argmin}\{\|X_\Lambda\beta -y\|_2: \beta \in \mathbb{R}^p\}
\end{equation*}
and set the other coordinates to zero. If the matrix $X_\Lambda^\top X_\Lambda$ is invertible then $\widehat{\beta}_\Lambda = (X_\Lambda^\top X_\Lambda)^{-1} X_\Lambda^\top y$. 

Putting all above discussion together, a complete two-stage procedure for finding a solution of problem~\eqref{problem:Dantzig-New} is described in Algorithm~\ref{alg:matrix-final}.

\begin{algorithm}[htb!]\caption{(Two-stage scheme for problem~\eqref{problem:Dantzig-New})}\label{alg:matrix-final}
 \begin{algorithmic}[htb!]
	\State \textbf{Input:} 
	Set the fixed parameters
	\[y\in\mathbb{R}^n,\quad A\in\mathbb{R}^{p\times p},\quad b \in\mathbb{R}^p, \quad \delta,\;\alpha,\; tol\in\mathbb{R}_+, \text{ and } \lambda = 0.999\alpha / \|A\|_2^2.\]
	\State \textbf{Initialization:}
	Set the initial parameters 
	\[\tau^0 = 0, \quad \beta^{-1} = \beta^0 = 0,\text{ and } k=0.\]

	\State \textbf{Stage-I:} Generate the sequence $\{(\tau^k,\beta^k):k\in\mathbb{N}\}$ using Equations~\eqref{eq:prox-L1} and~\eqref{eq:itr-tau-beta}.\\
	\While{ (stopping criterion not met)}\\
	\begin{eqnarray*}
		\begin{array}{l}
		\tau^{k+1}\leftarrow \mathrm{prox}_{\delta\|\cdot\|_1}(A (2\beta^{k}-\beta^{k-1})+\tau^k-b), \\
		\beta^{k+1}\leftarrow \mathrm{prox}_{\frac{1}{\alpha}\|\cdot\|_1}(\beta^{k}-\frac{\lambda}{\alpha} A^\top \tau^{k+1}),\\
		k\leftarrow k+1.
		\end{array}
	\end{eqnarray*}
	\EndWhile
	
	\State \textbf{Stage-II:}
	\noindent Let $\displaystyle \left(\tau^\infty, \beta^\infty\right)$ be the last set of parameters computed in Stage-I.  
	\begin{itemize}
		\item Approximate $\mathrm{supp}(\beta^\infty)$ by $\Lambda = \{j : | \beta^\infty(j) | < tol \}$.
		\item Compute  $\widehat{v} = \displaystyle \mathrm{argmin}_{v \in \mathbb{R}^{|\Lambda|}} \left\{ \|X_{\Lambda}v - y\|_2 \right\}$.
		\item Extend $\widehat{v}$ to form the Dantzig selector $\widehat{\beta}$ on $\Lambda$:
		\[ \begin{cases} \widehat{\beta}(\Lambda(i)) = \widehat{v}(i), & \text{ for } i = 1:|\Lambda|,\\
		\widehat{\beta}(j) = 0, &\text{ for } j\notin \Lambda. \end{cases} \]
	\end{itemize}
 \end{algorithmic}
\end{algorithm}

Stage-I of Algorithm~\ref{alg:matrix-final} terminates once the sequence $\{(\tau^k,\beta^k):k\in\mathbb{N}\}$ reaches a stationary point.  To estimate when this occurs, terminate the iterations when either of the following stopping criteria are met:
\begin{enumerate}
	\item The relative change between successive terms in the sequence $\{\beta^k\}$ falls below a specified tolerance;
	\[ \frac{\left\| \beta^{k+1} - \beta^k \right\|_2}{\|\beta^k\|_2} < \varepsilon,\]
	for some $\varepsilon>0$, or
	\item The support of the sequence $\{\beta^k\}$ is stationary for a specified number of successive iterations;
	\[ \mathrm{supp}(\beta^k) = \mathrm{supp}(\beta^{k+1}) = \cdots = \mathrm{supp}(\beta^{k+\eta}),\]
	for a fixed $\eta\in\mathbb{N}$ and some positive integer $k$.
\end{enumerate}

Stage-I is the largest contributor to the computational complexity of Algorithm~\ref{alg:matrix-final}, with each iteration having complexity $\mathcal{O}(4np)$.  In comparison, each outer loop of ADM computing the $\tau$ and $\gamma$-related subproblems has complexity $\mathcal{O}(4np)$, while each inner loop of ADM approximating the $\beta$-related subproblem has complexity $\mathcal{O}(8np)$.  In general, Algorithm~\ref{alg:matrix-final} and ADM will use a different number of iterations to terminate their iterative stages, so their overall complexities cannot be directly compared.  However, the numerical experiments in the next section indicate that Algorithm~\ref{alg:matrix-final} tends to have less overall complexity than ADM since Algorithm~\ref{alg:matrix-final} has a shorter runtime even in situations where it requires more iterations.

\section{Numerical Experiments}\label{sec:experiments}
In the following experiments, we apply the proposed proximity operator based approach presented in Algorithm~\ref{alg:matrix-final} and the alternating direction method (ADM) presented in~\cite{Lu-Pong-Zhang:CSDA:12} to solve the Dantzig selector problem~\eqref{problem:Dantzig} using both synthetic and real data sets.  The experiments using synthetic data are performed in MATLAB R2013a on single nodes of the Condor Supercomputer, hosted at AFRL/RIT Affiliated Resource Center.  The full capabilities of Condor were not taken advantage of; we ran the algorithms in serial using single nodes to emulate a typical high end consumer workstation.  Each utilized node is equipped with an Intel Xeon X5650 6 core CPU, with 2.67 GHz and 6$\times$8 GB RAM.  The experiments using the real data set are performed in MATLAB R2014a on a PC with an Intel Core i7-3630QM \@2.40 GHz processor and 16 GB RAM running Windows 7 Enterprise.

\begin{example} Synthetic Data Set \end{example}

In this series of simulations, sparse coefficient vectors are generated then recovered from noisy random linear observations using both Algorithm~\ref{alg:matrix-final} and ADM.  The parameters used are $n=720m,\; p = 2560m$ and $s = 80m$ for $m\in\{2,3,\ldots,10\}$, and $\sigma \in\{0.01,\; 0.05,\; 0.10,\; 0.15\}$ corresponding to 1\%, 5\%, 10\% and 15\% noise levels.  For each combination of $m$ and $\sigma$, 100 simulations each of Algorithm~\ref{alg:matrix-final} and ADM are performed.  All other parameters for ADM are selected following the guidelines in~\cite[Section 3]{Lu-Pong-Zhang:CSDA:12} and the parameters selected for the initialization stage of Algorithm~\ref{alg:matrix-final} are $tol = 2\sigma$, $\alpha = 0.2\|A\|_2^2$ and $\delta = \sigma\sqrt{2\log p}$. The parameters for the stopping criteria are $\varepsilon = 10^{-4}$ and $\eta = \max\left\{ \left\lceil 4\log(\alpha)\log(\sigma) + 2\alpha\right\rceil, \; 5 \right\}$.  The parameters $tol, \delta$ and $\eta$ depend on the noise level $\sigma$, which in practice may not be known a priori.  However, the noise level may be well-approximated using existing methods.  In the event that the noise level is not accurately approximated, the speed of convergence of Algorithm~\ref{alg:matrix-final} will be affected, but the accuracy should not suffer much.  The stopping criteria

The $n\times p$ sensing matrices $X$ are generated for each simulation with independent Gaussian entries normalized so each column has unit $\ell_2$ norm.  To generate the coefficient vector, for each simulation a support set $S$ of size $|S|=s$ is selected uniformly at random.  Then the vector $\beta$ with indices in $S$ is defined according to $\beta_{S(i)} = \epsilon_i(1+|a_i|)$, where $\{a_i\}$ is a collection of independently and identically distributed random variables sampled from the standard normal distribution and $\{\varepsilon_i\}$ is a collection of independently and identically distributed random variables sampled from the uniform distribution on $\{-1,1\}$.  For $i\notin S$, set  $\beta_i = 0$.  Then Algorithm~\ref{alg:matrix-final} and ADM are used to approximate the Dantzig selector $\widehat{\beta}$ from the observations $y = X\beta+z$, where $z$ is a collection of independent and identically distributed random variables sampled from the normal distribution with mean zero and standard deviation $\sigma$.

\begin{figure}[htb]
\centering
		\includegraphics[width=0.48\textwidth]{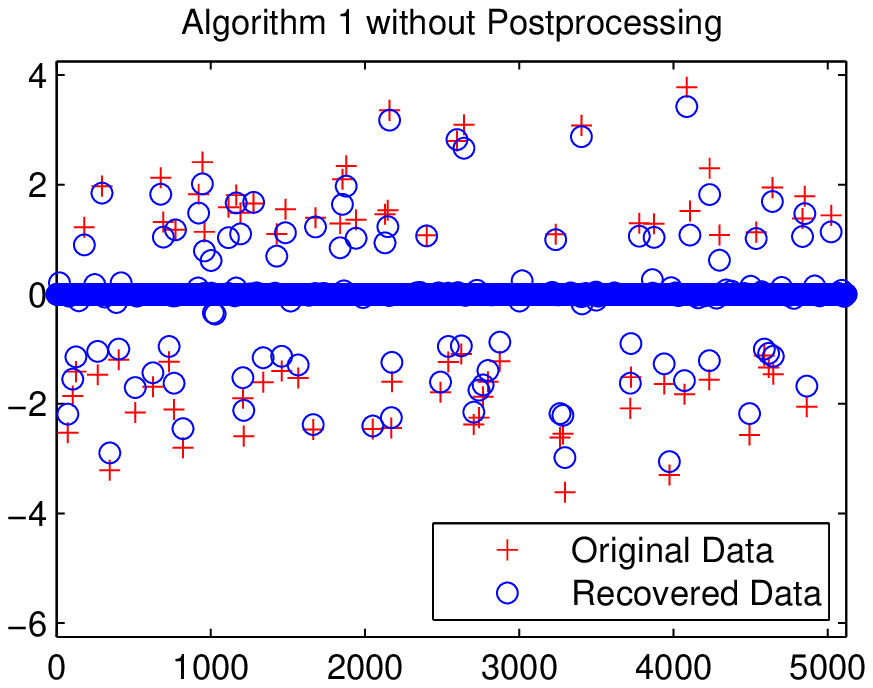}
		\includegraphics[width=0.48\textwidth]{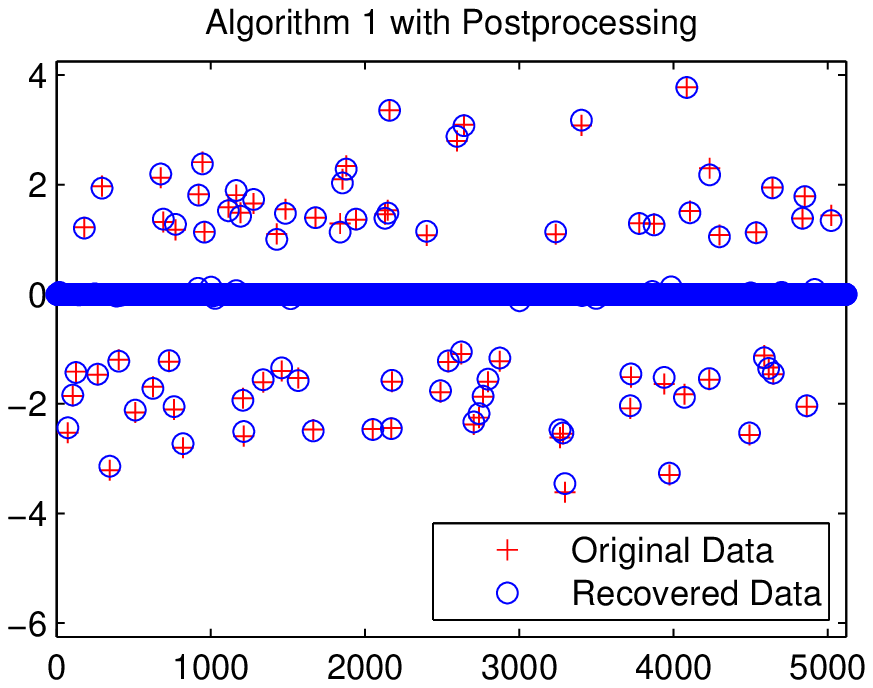}

		\includegraphics[width=0.48\textwidth]{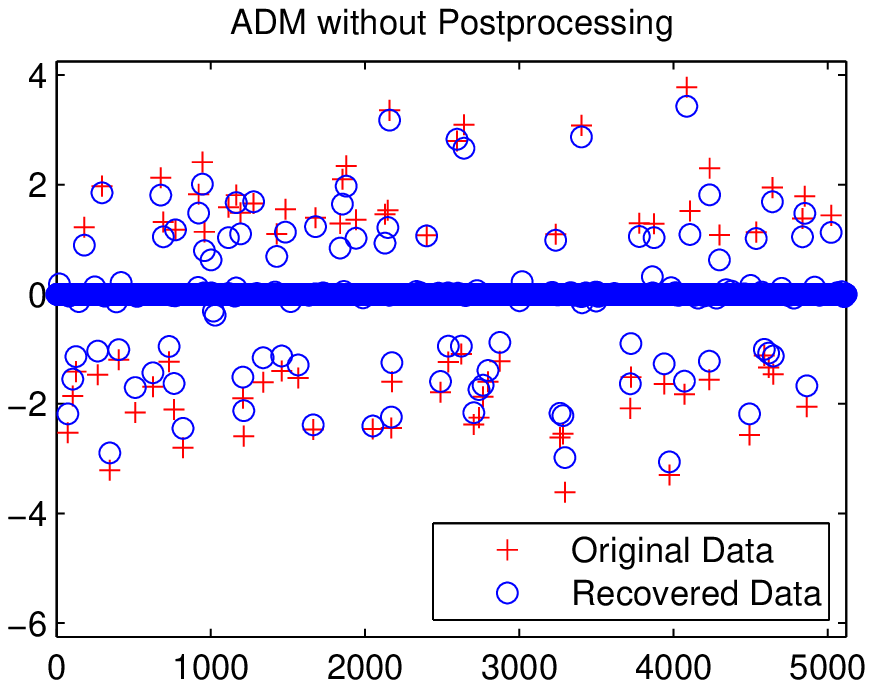}
		\includegraphics[width=0.48\textwidth]{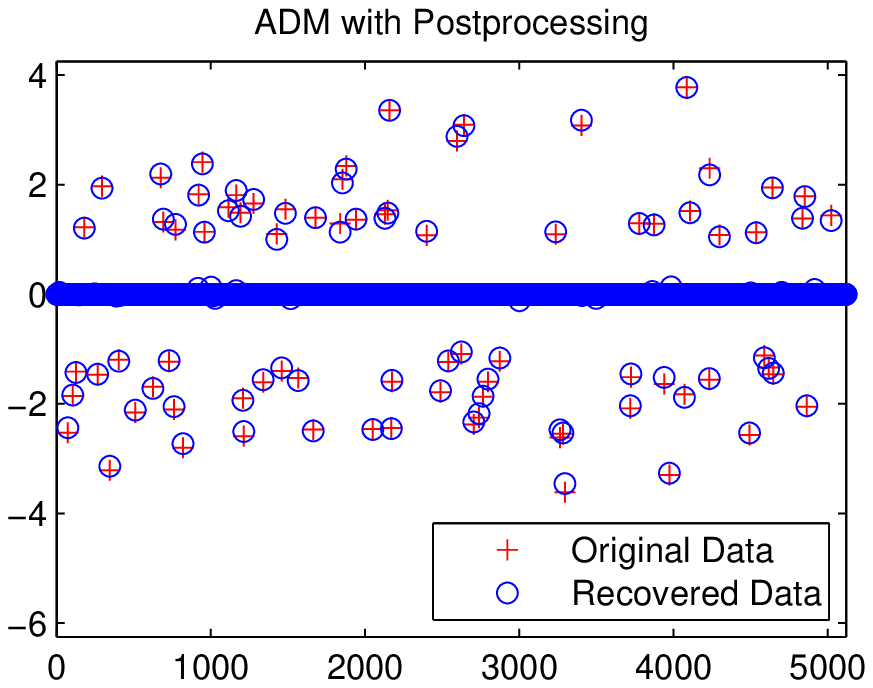}
\caption{A demonstration of the accuracy of the Dantzig selector recovered using Algorithm~\ref{alg:matrix-final} and ADM with and without postprocessing for a single simulation of Experiment 3.1 with parameters $\sigma = 0.05$ with $(n,p,s)=(720, 2560, 80)$. }
\label{accuracy plot med noise}
\end{figure}

The accuracy of the Dantzig selector recovered in the simulations is measured by
\begin{equation}\label{accuracy}
	\rho := \left(\frac{ \| \beta - \widehat{\beta}\|_2^2}{\sum_{j=1}^p \min\{ \beta_j^2,\sigma^2\}}\right)^{1/2},
\end{equation}
where $\beta$ denotes the true parameter and $\widehat{\beta}$ denotes the parameter recovered using either Algorithm~\ref{alg:matrix-final} or ADM.
The denominator term of Equation~\eqref{accuracy} is the expected mean squared-error of the ideal estimator \cite{Candes-Tao:AS:07}. Therefore, $\rho\geq 0$, and a smaller $\rho$ implies a more accurate estimator.

\begin{figure}[thb]
	\centering
		\includegraphics[width=0.48\textwidth]{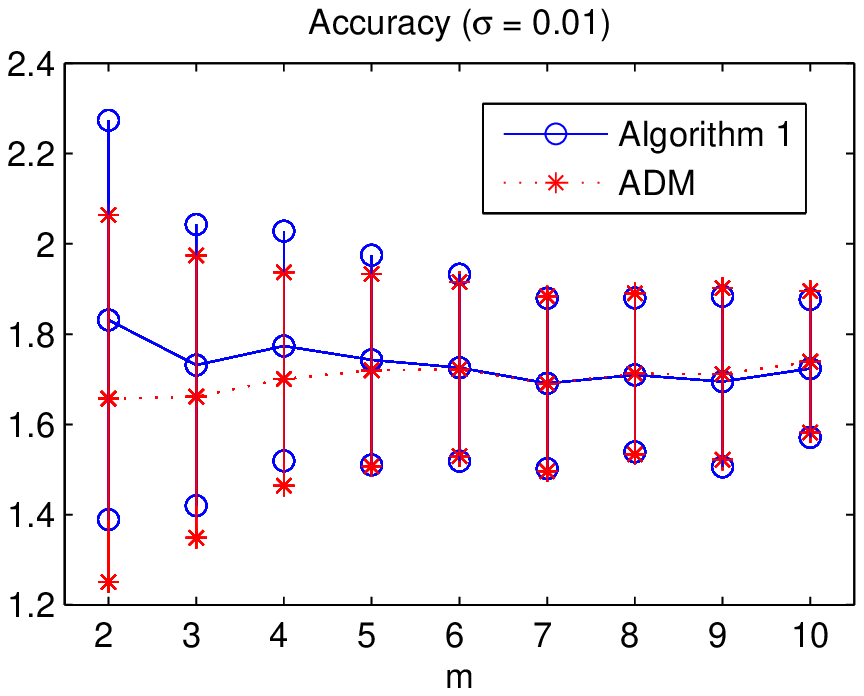}
		\includegraphics[width=0.48\textwidth]{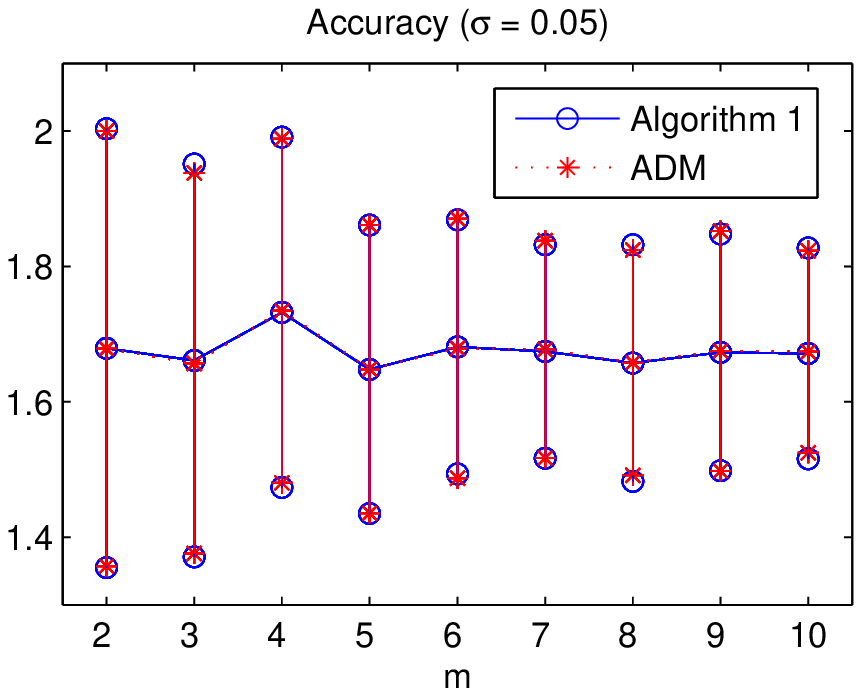}

		\includegraphics[width=0.48\textwidth]{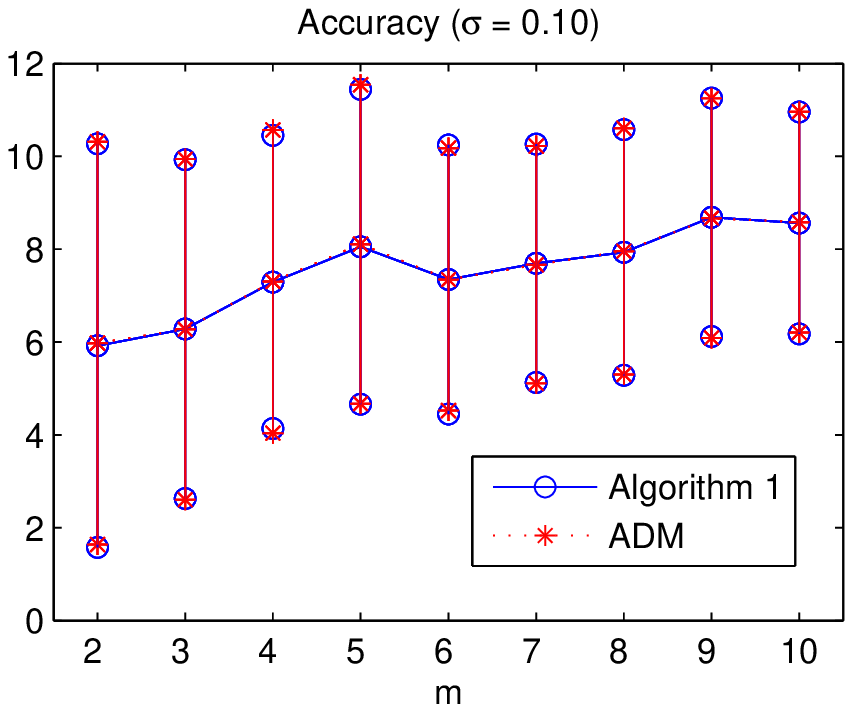}
		\includegraphics[width=0.48\textwidth]{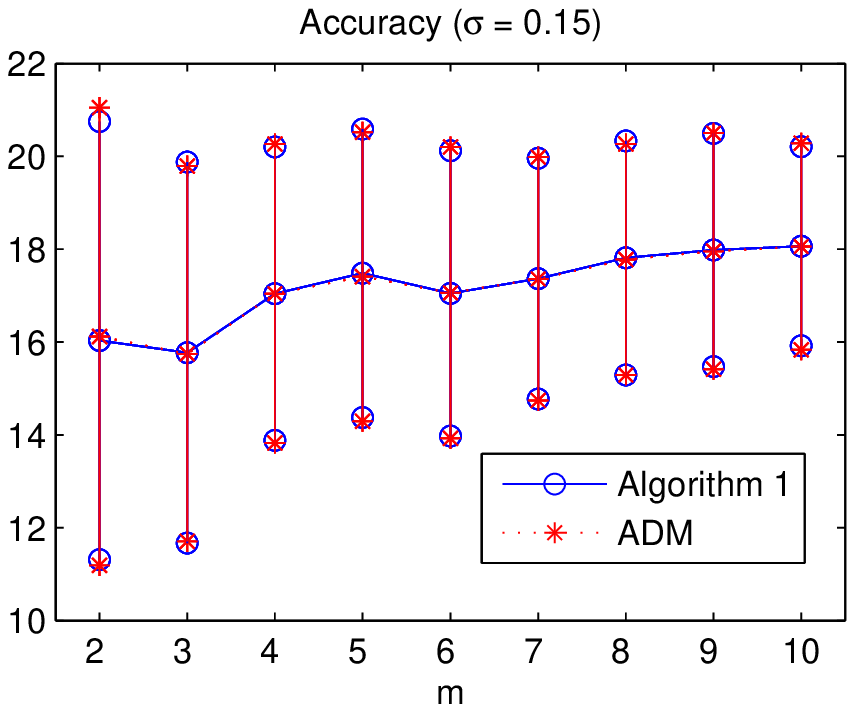}
	\caption{A comparison of $\rho$, computed as in Equation~\eqref{accuracy} which measures the accuracy of the approximated Dantzig selectors, for Algorithm~\ref{alg:matrix-final} and ADM for noise levels $\sigma = 0.01, 0.05, 0.10$ and $0.15$ in Example~3.1.  In each plot, the points along the curve represent the mean number of iterations required for each parameter $m$ over 100 simulations, and the points on the vertical lines represent one standard deviation away from the means. }
	\label{fig:accuracy}
\end{figure}

The effects of Stage-II of Algorithm~\ref{alg:matrix-final} and the postprocessing step of ADM are illustrated in Figure~\ref{accuracy plot med noise}.  The figure displays values of the exact simulated vector $\beta$ and of the Dantzig selector $\widehat{\beta}$ approximated by each algorithm,  
first without performing postprocessing (the left column of Figure~\ref{accuracy plot med noise}) and then with postprocessing (the right column of Figure~\ref{accuracy plot med noise}) for one simulation with parameters $(n,p,s)=(720, 2560, 80)$ and noise $\sigma=0.05$. One can clearly see that the postprocessing not only corrects the underestimated magnitudes of nonzero components of the estimates, but also eliminates unwanted nonzero components.

\begin{figure}[thb]
	\centering
		\includegraphics[width=0.48\textwidth]{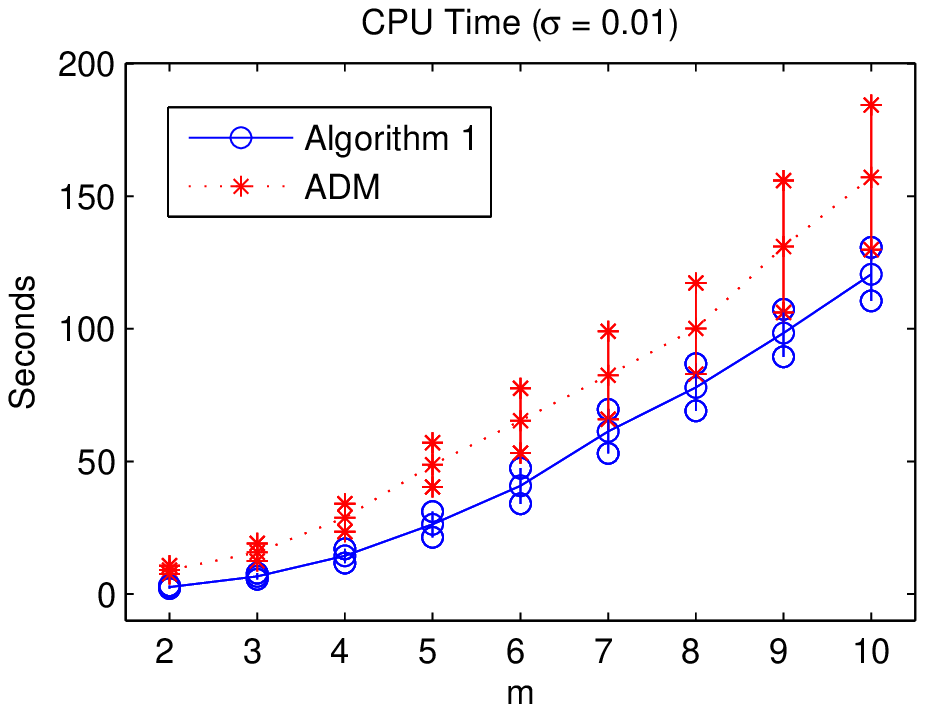}
		\includegraphics[width=0.48\textwidth]{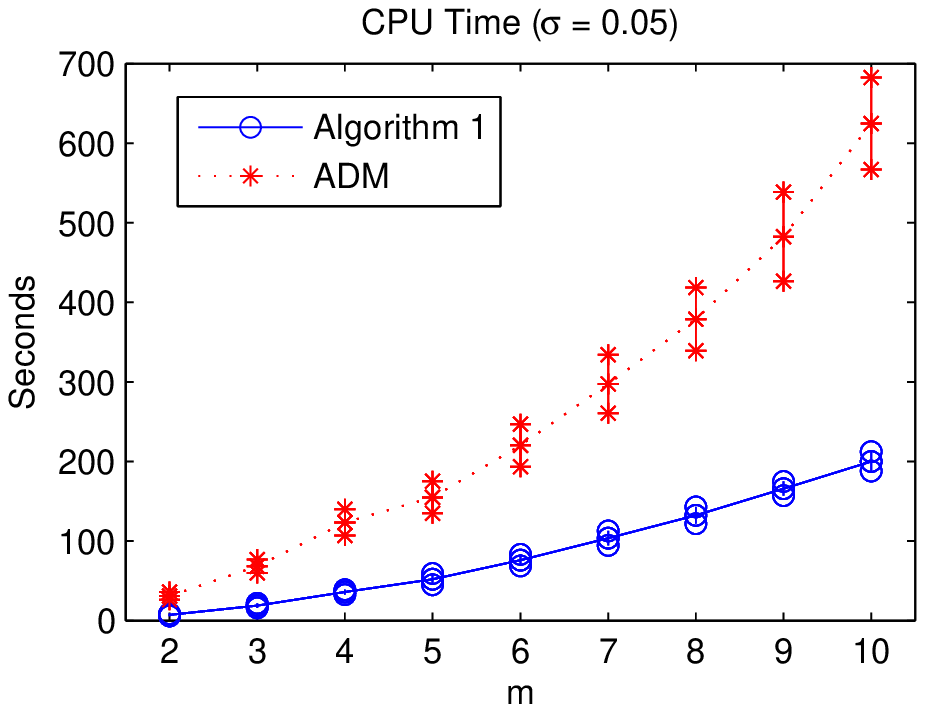}

		\includegraphics[width=0.48\textwidth]{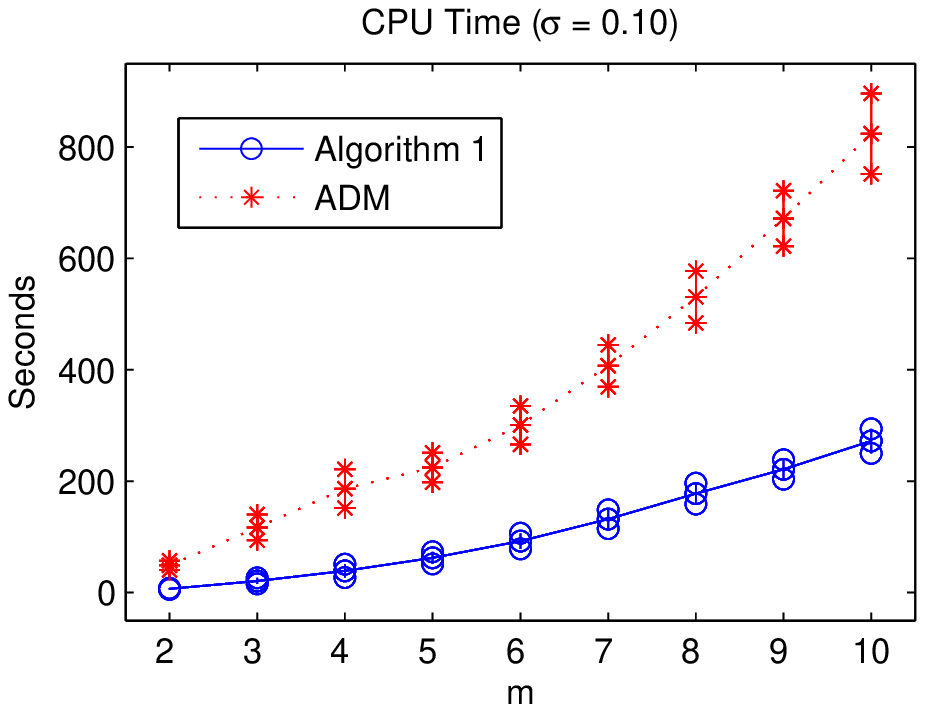}
		\includegraphics[width=0.48\textwidth]{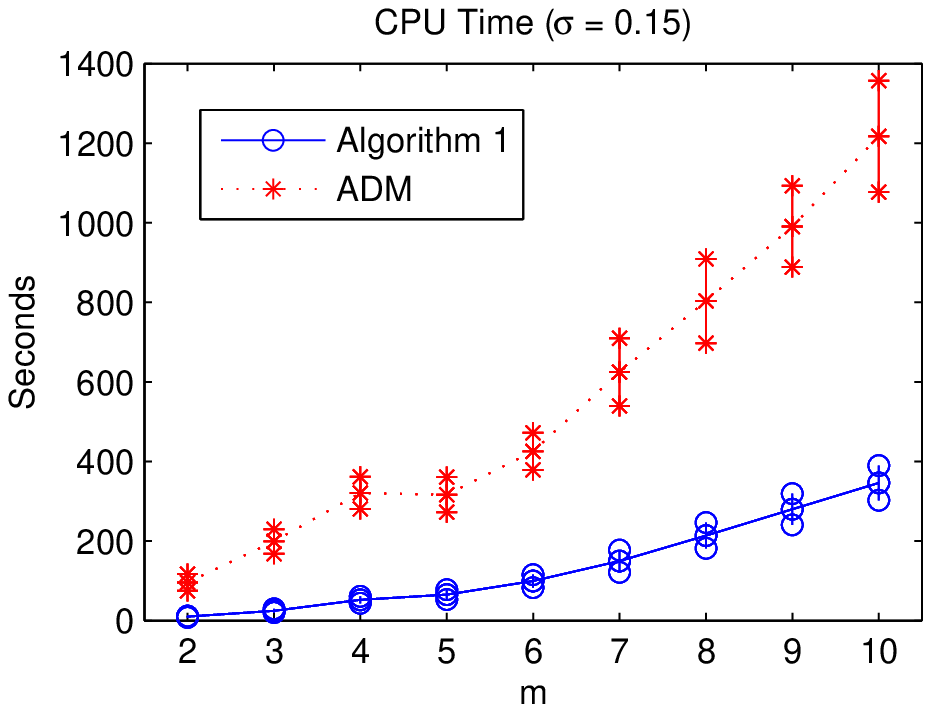}

	\caption{A comparison of the CPU time required to recover the Dantzig selector using Algorithm~\ref{alg:matrix-final} and ADM for noise levels $\sigma = 0.01, 0.05, 0.10$ and $0.15$ as in Example~3.1.  In each plot, the points along the curve represent the mean number of iterations required for each parameter $m$ over 100 simulations, and the points on the vertical lines represent one standard deviation away from the means.}
	\label{fig:time}
\end{figure}

\begin{figure}[thb]
	\centering
		\includegraphics[width=0.48\textwidth]{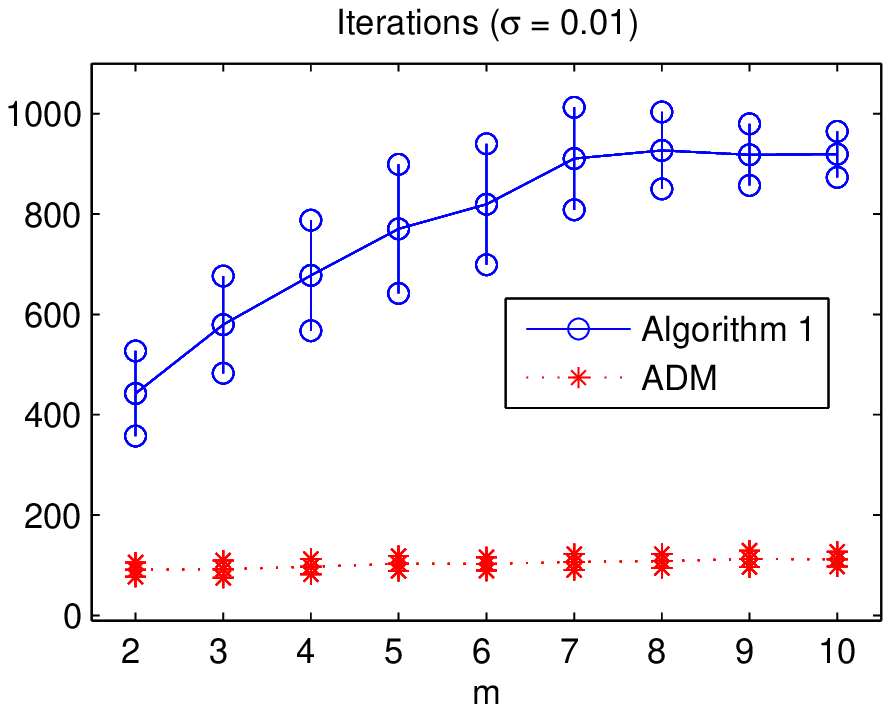}
		\includegraphics[width=0.48\textwidth]{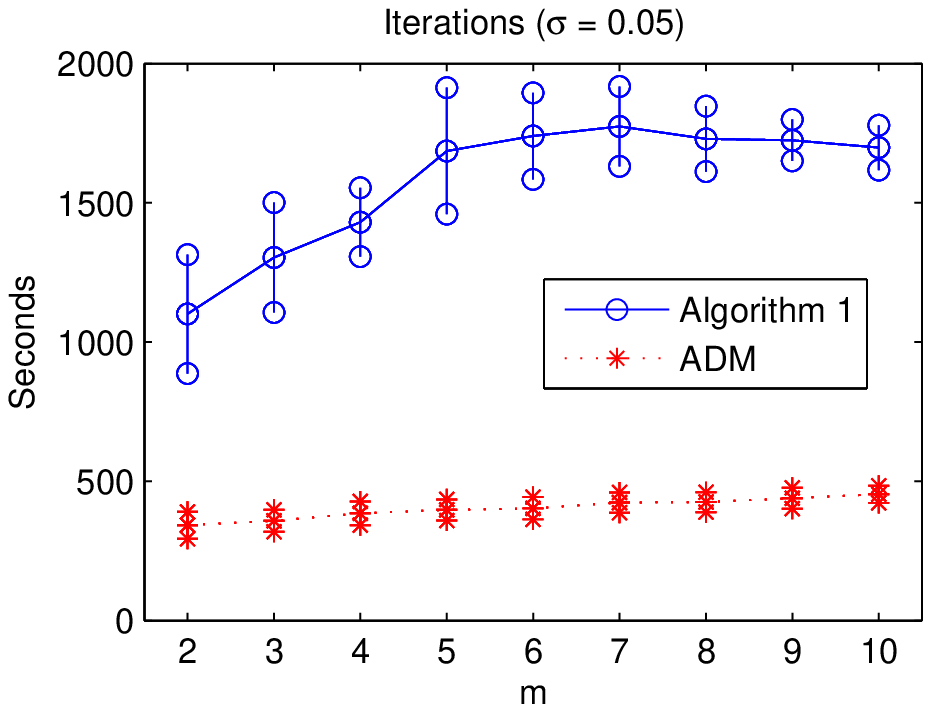}

		\includegraphics[width=0.48\textwidth]{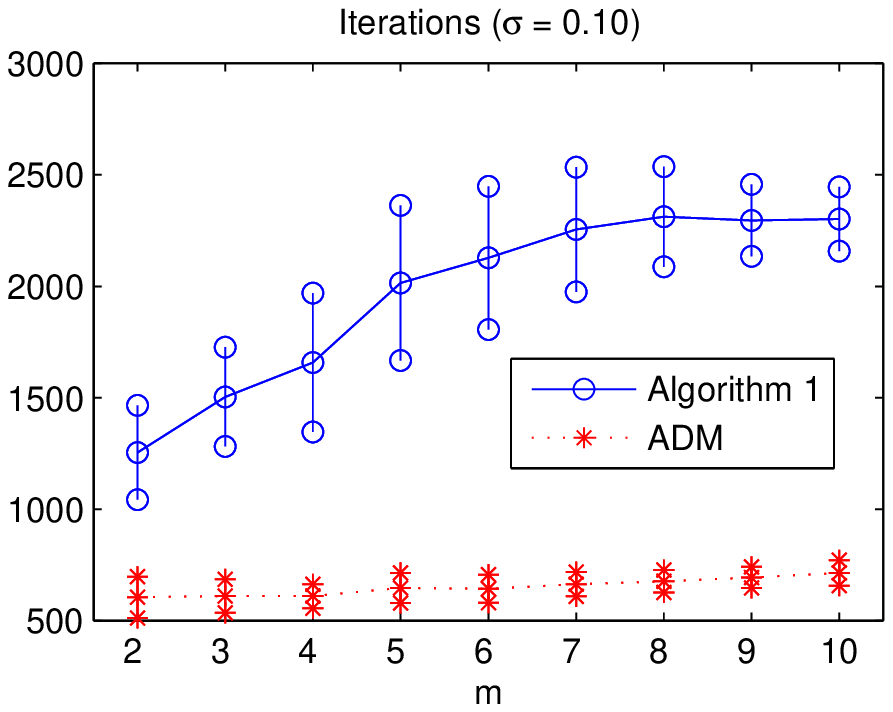}
		\includegraphics[width=0.48\textwidth]{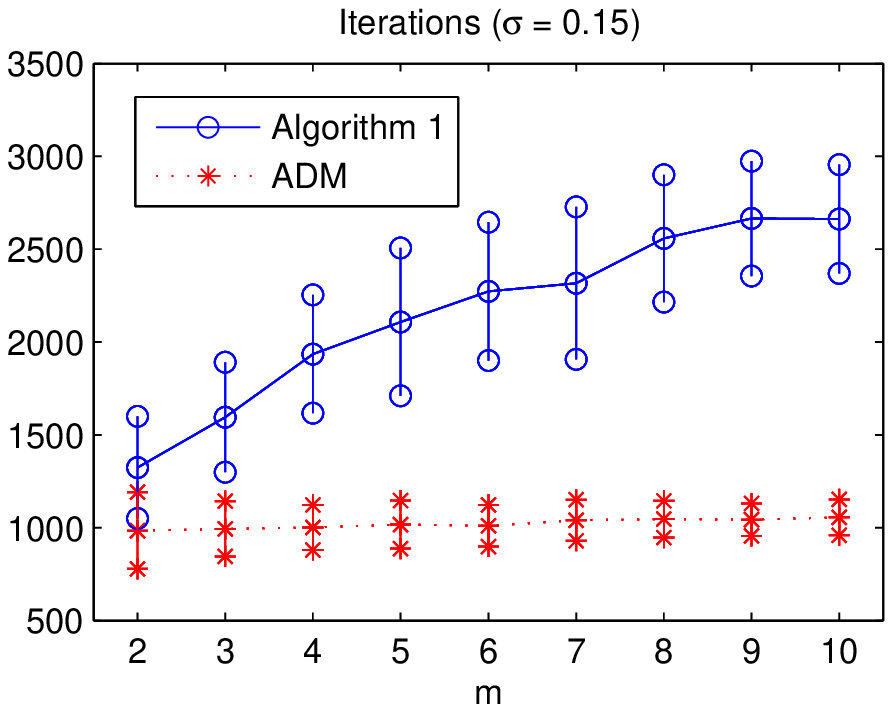}

	\caption{A comparison of the number of iterations required to recover the Dantzig selector using Algorithm~\ref{alg:matrix-final} and ADM for noise levels $\sigma = 0.01, 0.05, 0.10$ and $0.15$ in Example~3.1.  In each plot, the points along the curve represent the mean number of iterations required for each parameter $m$ over 100 simulations, and the points on the vertical lines represent one standard deviation away from the means.}
	\label{fig:iterations}
\end{figure}

The results of the above simulations suggest that Algorithm~\ref{alg:matrix-final} has less overall complexity than ADM, since the accuracy of the Dantzig selectors approximated by each method are similar yet Algorithm~\ref{alg:matrix-final} completes much faster than ADM, even when requiring more iterations.
Figure~\ref{fig:accuracy} displays the mean and standard deviation of $\rho$ over 100 simulations for each parameter $m$ and $\sigma$ and for both Algorithm~\ref{alg:matrix-final} and ADM. Note that the accuracy of the Dantzig selector approximated by the two algorithms are very similar across all parameter levels.
Figure~\ref{fig:time} displays the mean and standard deviation of the CPU time, and Figure~\ref{fig:iterations} displays the mean and standard deviation of the total number of iterations performed by Algorithm~\ref{alg:matrix-final} and the total number of iterations performed in the inner loop of ADM for 100 simulations for each parameter $m$ and $\sigma$.
From the figures, one can see that although Algorithm~\ref{alg:matrix-final} requires more iterations than ADM, Algorithm~\ref{alg:matrix-final} completes significantly faster.


\begin{example} Leukemia Data Set \end{example}

In this experiment, the Dantzig selectors produced by Algorithm~\ref{alg:matrix-final} and by ADM are used with a collection of biomarker data to indicate whether a patient may be diagnosed with a specific type of cancer.  The biomarker dataset, first introduced in~\cite{Golub} and studied in~\cite{Tibshirani et al,Wang-Yuan:SISC:12}, contains the measurements of 7128 genes related to leukemia diagnoses.  The dataset is split into a training set and a testing set.  The training set is sampled from 38 patients, 27 of whom were diagnosed with acute lymphocytic leukemia (ALL) and 11 with acute mylogenous leukemia (AML).  The testing set is sampled from 34 patients, 20 diagnosed with ALL and 14 with AML.  

Let $X_\text{train}\in\mathbb{R}^{38\times 7128}$ contain the biomarker data in the training set, where each row is all~7128 gene measurements of a single patient and each column has been normalized to have unit $\ell_2$ norm. Let $y_\text{train} \in\mathbb{R}^{38}$ be the column vector indicating the diagnosis of each patient in the training set:
\begin{equation*}
	y_\text{train}(j) = \begin{cases}0, &\text{ if patient $j$ in the training set is diagnosed with ALL}, \\
	1, &\text{ if patient $j$ in the training set is diagnosed with AML}.
\end{cases}
\end{equation*}
Similarly define $X_\text{test} \in \mathbb{R}^{34\times 7128}$ and $y_\text{test} \in\mathbb{R}^{34}$ from the data in the testing set. 

This experiment has a training phase and a testing phase.  In the training phase, a sparse vector $\widehat{\beta}$ is found such that $X_\text{train}\widehat{\beta} = y_\text{train}$.  To preprocess the data, only the biomarkers with the largest variance are used to train the parameter $\widehat{\beta}$.  To this end, select a positive integer $N$, and let $\Lambda$ be the $N$ indices of columns from $X_\text{train}$ with largest variance.  Let $\tilde{X}_\text{train} \in \mathbb{R}^{38\times N}$ be the submatrix of $X_\text{train}$ with columns in $\Lambda$.  Form the reduced problem
\begin{equation}\label{eq:leukemia training reduced}
	\widehat{\beta}_\Lambda \in \mathrm{argmin}_{\beta \in \mathbb{R}^N} \left\{ \|\beta\|_1 : \; \left\| \tilde{X}^\top_\text{train} \left(\tilde{X}_\text{train}\beta - y_\text{train}\right)\right\|_\infty \leq \delta \right\}.
\end{equation}
The Dantzig selector $\widehat{\beta}_\Lambda \in \mathbb{R}^{N}$ satisfying problem~\eqref{eq:leukemia training reduced} is computed using Algorithm~\ref{alg:matrix-final} and ADM, then extended to form  $\widehat{\beta}\in\mathbb{R}^{7128}$ via
	\begin{equation*}\begin{cases}
		\widehat{\beta}(\Lambda(j)) = \widehat{\beta}_\Lambda(j), &\text{ for } j = 1:N,\\
		\widehat{\beta}(k) = 0, &\text{ if } k\notin \Lambda.
	\end{cases}\end{equation*}

In the testing phase, the trained parameter $\widehat{\beta}$ is used to predict the diagnoses of patients in the testing set.  The predictive indicator vector $\widehat{y}_\mathrm{test}\in\mathbb{R}^{34}$ is computed from $y=X_\text{test}\widehat{\beta}$ by thresholding and clustering values near the threshold boundary.
Set 
\begin{equation*}
	\widehat{y}_\text{test}(j) = \begin{cases}
		0, &\text{ if } y(j) < 0.49, \\
		1, &\text{ if } 0.51< y(j). \\
	\end{cases}
\end{equation*}
Let $y_0 = \max \{ y(j) : y(j) < 0.49\}$ and $y_1 = \min \{y(j) : 0.51 < y(j)\}$.  For values of $j$ such that $0.49\leq y(j)\leq0.51$, set
\begin{equation*}
	\widehat{y}_\text{test}(j) = \begin{cases}
		0, &\text{ if } |y(j)-y_0| \leq |y(j) - y_1|, \\
		1, &\text{ if } |y(j)-y_1|<|y(j)-y_0|.\\
	\end{cases}
\end{equation*}
The $j^\text{th}$ patient in the testing set is predicted to have a diagnosis of ALL if $\widehat{y}_\text{test}(j) = 0$ and a diagnosis of AML if $\widehat{y}_\text{test}(j) = 1$. 

\begin{figure}
	\centering
	\begin{subfigure}[t]{0.48\textwidth}
		\includegraphics[width=\textwidth]{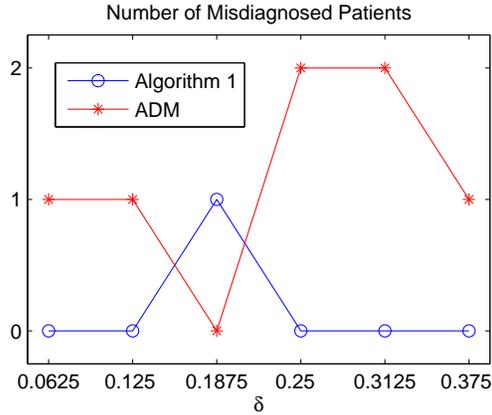}
		\caption{The number of patients in the testing set misdiagnosed by the predicted indicator vector recovered by Algorithm~1 and ADM for various values of the parameter $\delta$. }
	\end{subfigure}
	\begin{subfigure}[t]{0.48\textwidth}
		\includegraphics[width=\textwidth]{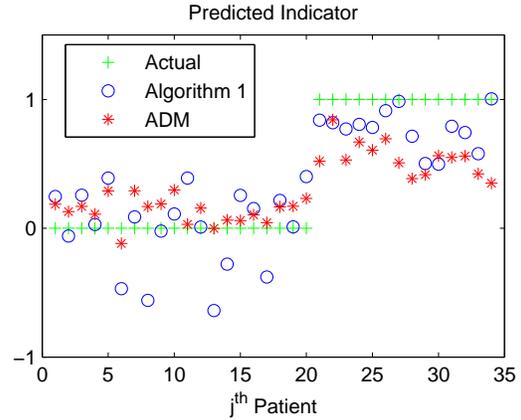}
		\caption{Values of the actual diagnosis indicator vector $y_\text{test}$ along with values of the predicted indicator vectors recovered by Algorithm~1 and ADM prior to separating values into classification groups for $\delta = 0.25$.}
		\label{fig:predictor}
	\end{subfigure}

	\begin{subfigure}[t]{0.48\textwidth}
		\includegraphics[width=\textwidth]{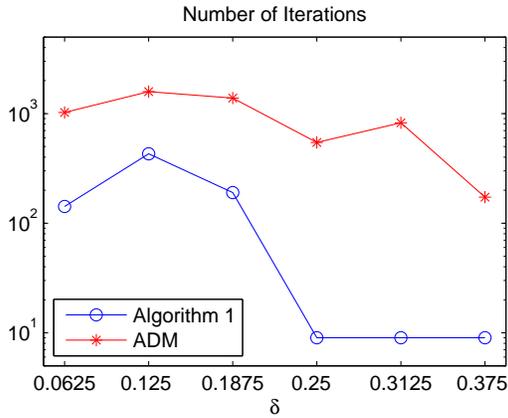}
		\caption{The number of iterations required by Algorithm~\ref{alg:matrix-final} and ADM to recover the Dantzig selector $\widehat{\beta}_\Lambda$ for various values of the parameter $\delta$.}
	\end{subfigure}
	\begin{subfigure}[t]{0.48\textwidth}
		\includegraphics[width=\textwidth]{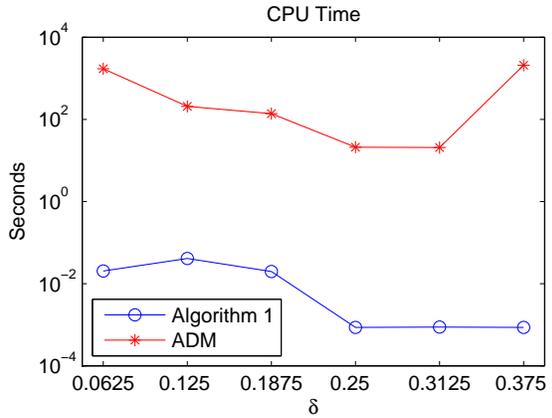}
		\caption{The CPU runtime required by Algorithm~\ref{alg:matrix-final} and ADM to recover the Dantzig selector $\widehat{\beta}_\Lambda$ for various values of the parameter $\delta$.}
	\end{subfigure}
	\caption{Plots regarding the indicator vector, used to predict a leukemia diagnosis in patients in the testing set as in Example 3.2, using both Algorithm~\ref{alg:matrix-final} and ADM.  }
	\label{fig:leukemia}
\end{figure}

The above procedure was used to predict the diagnoses of patients in the testing set using the Dantzig selector $\widehat{\beta}_\Lambda$ computed  using both Algorithm~\ref{alg:matrix-final} and ADM with parameters $N=1000$, $\alpha=\|X_\text{train}^\top X_\text{train}\|^2_2$ and $tol = 0.1$ and stopping criteria parameters $\eta=80$ and $\varepsilon = 10^{-4}$ for each $\delta$ in $\{0.0625,\;0.125,\;0.1875,\;0.25,\;0.3125,\;0.375\}$. Figure~\ref{fig:leukemia} displays the results of these simulations regarding the accuracy of the recovered indicator vector $\widehat{y}_\text{test}$ in predicting the leukemia diagnoses of patients in the testing set, as well as the number of iterations and CPU runtime used by Algorithm~\ref{alg:matrix-final} and ADM.  As shown in Figure~\ref{fig:leukemia}(a), Algorithm~\ref{alg:matrix-final} typically predicted the diagnoses of patients with higher acuracy than ADM.  Moreover, for each parameter $\delta$, Algorithm~\ref{alg:matrix-final} used fewer iterations than ADM and the time used by Algorithm~\ref{alg:matrix-final} was several orders of magnitude less than the time used by ADM, as shown in Figures~\ref{fig:leukemia}(c) and (d).  Figure~\ref{fig:leukemia}(b) illustrates the tendency of Algorithm~\ref{alg:matrix-final} to predict the diagnosis of patients in the testing set with higher accuracy than ADM.  This plot displays the values of $y = X_\text{test}\widehat{\beta}$ recovered using Algorithm~\ref{alg:matrix-final}  and by ADM prior to the thresholding step, along with the true values of $y_\text{test}$.  Since the values recovered by Algorithm~\ref{alg:matrix-final} tend to be more spread out, it is easier to accurately separate them into two distinct clusters.

\section{Conclusion}\label{sec:conclusion}
In this paper, we have developed an iterative algorithm to compute the Dantzig selector, the solution to the minimization problem in problem \eqref{problem:Dantzig}.  The algorithm is based on the proximity operator and its relationship to problem \eqref{problem:Dantzig-New}.  The two-stage algorithm we proposed is an improvement over some other recently proposed methods to find the Dantzig selector, which require the use of inner loop to estimate parameters within each step of the algorithm.  Additionally, our proposed method uses a novel stopping criterion based upon the support of the approximated parameters.   

We compare the proposed algorithm to the alternating direction method proposed in~\cite{Lu-Pong-Zhang:CSDA:12}. Theoretically, two methods produce results of similar quality, however each iteration of Stage-I of Algorithm~\ref{alg:matrix-final} has less computational complexity than each iteration of the inner loop of the alternating direction method.  The numerical experiments demonstrate that the proposed method and the alternating direction method typically approximate the Dantzig selectors with similar accuracy, yet Algorithm~\ref{alg:matrix-final} produces results in significantly less time, whether it uses more iterations than the alternating direction method, as in Experiment~3.1, or fewer iterations than the alternating direction method, as in Experiment~3.2.


\section*{Acknowledgements}
The authors are grateful to the anonymous reviewers for their helpful comments.  The authors also would like to thank Drs.\ X.\ Wang and X.\ Yuan for providing the MATLAB code to approximate the Dantzig selector using the alternating direction method and for sharing the real dataset used in Example 3.2.

\bibliographystyle{siam}

\end{document}